\documentclass[12pt]{amsart}
\usepackage[english]{babel}
\usepackage{graphicx,epsfig,caption,subcaption}
\usepackage[T1]{fontenc}
\usepackage{fullpage}
\usepackage{color}
\usepackage{amsmath,amscd,amsthm,amsfonts,latexsym,soul}

\theoremstyle{plain}
\newtheorem{theorem}                 {Theorem}      [section]
\newtheorem{proposition}  [theorem]  {Proposition}

\newtheorem{lemma}        [theorem]  {Lemma}

\theoremstyle{definition}
\newtheorem{example}      [theorem]  {Example}

\newtheorem{definition}   [theorem]  {Definition}

\def \H{\mbox{$\h_3(\tau)$}}

\def \r{\mbox{${\mathbb R}$}}

\def \h{\mbox{${\mathbb H}$}}
\def \E{\mbox{${\mathbb E}$}}
\def \s{\mbox{${\mathbb S}$}}
\def \n{\mbox{${{\nabla}}^\tau$}}
\def \R{\mbox{${\mathrm R^\tau}$}}
\def \g{\mbox{$g_\tau$}}
\def \L{\mbox{${\mathbb L}$}}
\begin{document}

\title{Constant angle surfaces in the Lorentzian Heisenberg group}

\author{Irene I. Onnis}
\address{Departamento de Matem\'{a}tica, C.P. 668\\ ICMC,
USP, 13560-970, S\~{a}o Carlos, SP\\ Brasil}
\email{onnis@icmc.usp.br}

\author{P. Piu}
\address{Universit\`a degli Studi di Cagliari\\
Dipartimento di Matematica e Informatica\\
Via Ospedale 72\\
09124 Cagliari}
\email{piu@unica.it}

\thanks{The first author was supported by grant 2015/00692-5, S\~ao Paulo Research Foundation (Fapesp). The second author was supported by PRIN 2015 ``Varietà reali e complesse: geometria, topologia e analisi armonica''  Italy; and GNSAGA-INdAM, Italy.}

\begin{abstract}
In this paper, we define and, then, we characterize constant angle  spacelike and timelike  surfaces in the three-dimensional Heisenberg group, equipped with a $1$-parameter family of Lorentzian metrics. In particular, we give an explicit local parametrization of these surfaces and we produce some examples.
\end{abstract}

\maketitle

\section{Introduction}

In recent years much work has been done to understand the geometry of surfaces whose unit normal vector field forms a constant angle with a fixed field of directions of the ambient space. These surfaces are called {\em helix surfaces} or {\em constant angle surfaces} and they have been studied in all the $3$-dimensional geometries.
Several classification results were obtained so far, in different ambient spaces and, among them, we mention \cite{CDS07, DM09, DFVV07, Di,	FMV11, LM11, MO, MOP}. Moreover, helix submanifolds have been studied in higher dimensional euclidean spaces and product spaces in \cite{DSRH10,DSRH09,RH11}. 

In the case of the Riemannian Bianchi-Cartan-Vranceanu (BCV) spaces $\E^3(\kappa,\tau)$, as they admit a Riemannian submersion onto a surface of constant Gaussian curvature (called the Hopf fibration), it was considered the angle $\vartheta$ that the unit normal vector field of a surface in a BCV-space forms with the vector field tangent to the fibers of the Hopf fibration. This angle $\vartheta$ has a crucial role in the study of surfaces in BCV-spaces as shown by Daniel, in \cite{B}, where he proved that the equations of Gauss and Codazzi are given in terms of the function $\nu=\cos\vartheta$ and that this angle is one of the fundamental invariants for a surface in $\E^3(\kappa,\tau)$.

Concerning the study of helix surfaces in the Lorentzian BCV-spaces $\L^3(\kappa,\tau)$, that are described by the $2$-parameter family of Lorentzian metrics:
\begin{equation*}
g_{\kappa,\tau} =\frac{dx^{2} + dy^{2}}{F^{2}} -  \left(dz -
\tau\, \frac{y\,dx - x\,dy}{F}\right)^{2},\quad F(x,y)=1+\dfrac{\kappa}{4}(x^2+y^2),\quad \kappa, \tau\in\r,
\end{equation*}
defined on $\Omega\times\r$, with $\Omega=\{(x,y)\in\r^2\colon F(x,y)> 0\}$, we refer \cite{FN} and \cite{LM}. In \cite{LM}, the authors classified  constant angle spacelike surfaces in the Lorentz-Minkowski $3$-space, while in \cite{FN} are considered  constant angle spacelike and timelike  surfaces in the Lorentzian product spaces given by $\s^2\times \r_1$ and $\h^2 \times \r_1$. 

We observe that the projection map
$\pi: \L^3(\kappa,\tau) \to M^2(\kappa)$, given by $\pi(x, y, z) = (x, y)$, is a Riemannian submersion
from $\L^3(\kappa,\tau)$ to the surface of constant curvature $\kappa$ given by 
$$M^2(\kappa)=\Big(\Omega,\frac{dx^{2} + dy^{2}}{F^{2}}\Big)$$ and, also, its fibers are the
integral curves of the unit Killing vector field $\partial_z$, which is vertical with respect to $\pi$. The constant $\tau$ is called the {\it bundle curvature parameter} of the ambient spaces $\L^3(\kappa,\tau)$ and it satisfies the geometric identity:
\begin{equation}\label{eqprima}
{\nabla^{k,\tau}}_{X}\partial_z=\tau\, X\wedge \partial_z, \qquad X\in \mathfrak{X}(\L^3(\kappa,\tau)),
\end{equation}
 where ${\nabla^{k,\tau}}$ and $\wedge$ denote, respectively, the Levi-Civita connection and  the cross product of $\L^3(\kappa,\tau)$.
 
This paper is devoted to the study and the characterization of spacelike and timelike helix surfaces in the Lorentzian Heisenberg group given by $\L^3(0,\tau)$ (with $\tau\neq 0$), denoted by $\h_3(\tau)$, whose geometry we shall describe in Section~\ref{preli}. In Section~\ref{tre} we determine the Gauss and Codazzi equations of an oriented pseudo-Riemannian surface ${\mathcal M}$ immersed in $\h_3(\tau)$, proving that they involve the metric of ${\mathcal M}$, its
shape operator $A$, the tangential projection $T$ of the vertical vector field $\partial_z$ and the  function $\nu:=\g(N,\partial_z)\,\g(N,N)$, where $\g:=g_{0,\tau}$ and $N$ is the unit normal to ${\mathcal M}$. Moreover, from the equation~\eqref{eqprima} derive two additional equations (see \eqref{eq3} and \eqref{eq4}) that are used to determine the shape operator and the Levi-Civita connection of ${\mathcal M}$.

In Sections~\ref{spacelike} and \ref{timelike}  we define, respectively, the  constant angle spacelike and timelike surfaces in $\h_3(\tau)$ and we show that these surfaces have constant Gaussian
curvature. Finally, in the Theorems~\ref{teo-one} and \ref{teo-two} we establish the complete classification of these surfaces and, then, we construct some examples.

\section{Preliminaries}\label{preli}
Let $\h_3(\tau)$ (with $\tau\neq 0$) denote the $3$-dimensional Heisenberg group given by $\r^3$ equipped with
 the $1$-parameter family of Lorentzian metrics 
$$
g_\tau=dx^2+dy^2-(dz-\tau\,(y\,dx-x\,dy))^2,
$$
which renders the map $\pi:\h_3(\tau)\to\r^2$ a Riemannian submersion. With respect to this metric, the vector fields given by:
\begin{equation}\label{eq-basis}
\left\{\begin{aligned}
    E_1&=\frac{\partial}{\partial x}+\tau\, y\,\frac{\partial}{\partial z},\\
     E_2&=\frac{\partial}{\partial y}-\tau\, x\,\frac{\partial}{\partial z},\\
     E_3&=\frac{\partial}{\partial z},
     \end{aligned}\right.
\end{equation}
form a Lorentzian orthonormal basis on $\h_3(\tau)$ and the associated Levi-Civita connection $\n$, where $\n=\nabla^{0,\tau}$, is given by:
\begin{equation}
\begin{aligned}
&\n_{E_{1}}E_{1}= \n_{E_{2}}E_{2}= \n_{E_{3}}E_{3}=0,\\
&\n_{E_{2}}E_{1}=\tau\, E_{3}= -\n_{E_{1}}E_{2},\\
 &\n_{E_{3}}E_{1}=-\tau\, E_{2}= \n_{E_{1}}E_{3},\\
 &\n_{E_{3}}E_{2}=\tau\, E_{1}=\n_{E_{2}}E_{3}.
\end{aligned}
\label{nabla}
\end{equation}
We observe that $E_3$ is a (timelike) unit Killing vector field, that is tangent to the fibers of the submersion $\pi$ and it satisfies the following identity:
\begin{equation}\label{principal}
\n_{X}E_{3}=\tau\, X\wedge E_3, \qquad X\in \mathfrak{X}(\h_3(\tau)),
\end{equation}
where $\wedge$ is the cross product in $\h_3(\tau)$ defined by the formula
$$U\wedge V = (u_2\,v_3-u_3\,v_2)\,E_1-(u_1\,v_3-u_3\,v_1)\,E_2- (u_1\,v_2- u_2\,v_1)\,E_3.$$ Also, using the following convention 
$$\R(X,Y)Z=\n_{X}\n_{Y}Z-\n_{Y}\n_{X}Z-\n_{[X,Y]}Z,$$ the non zero components of the Riemann curvature tensor are:
\begin{equation}
\begin{aligned}
&\R(E_1,E_2)E_1=-3\tau^2\,E_2,\qquad \R(E_1,E_3)E_1=\tau^2\,E_3,\\
&\R(E_1,E_2)E_2=3\tau^2\,E_1,\,\;\quad \quad\; \R(E_2,E_3)E_2=\tau^2\,E_3,\\
&\R(E_2,E_3)E_3=\tau^2\,E_2,\quad\; \qquad \R(E_1,E_3)E_3=\tau^2\,E_1.\\
\end{aligned}
\end{equation}
Moreover, the tensor $\R$ can be described as we have done in the following result.
\begin{proposition}\label{tensor-R}
The Riemann curvature tensor $\R$ of $\H$ is determined by
\begin{equation}\label{eq-tensor-R}
\begin{aligned}
\R(X,Y)Z&=3\tau^2\,[\g(Y,Z)\,X-\g(X,Z)\,Y]\\
&+4\tau^2\,[\g(Y,E_3)\,\g(Z,E_3)\,X-\g(X,E_3)\,\g(Z,E_3)\,Y\\
&+\g(Y,Z)\,\g(X,E_3)\,E_3-\g(X,Z)\,\g(Y,E_3)\,E_3],
\end{aligned}
\end{equation}
for all vector fields $X,Y,Z$ on $\H$.
\end{proposition}
\begin{proof}
Putting  $\R(X\wedge Y,Z\wedge W)=\g(\R(X,Y)Z,W)$,
the matrix of $\R$ with respect to the basis $\{E_2\wedge E_3,E_3\wedge E_1,E_1\wedge E_2\}$ is given by:
$$\R=\left(\begin{array}{ccc}
-\tau^2&0&0\\
0&-\tau^2&0\\
0&0&-3\,\tau^2
\end{array}\right).$$ 
Now, we set
$X=\overline{X}+x\,\,E_3$,
where $\overline{X}$ is the  horizontal component of $X$ and $x=-\g(X,E_3)$, etc. Therefore, we obtain that
 \begin{align*}
    \g(\R(X,Y)\,Z,W)&=\g(\R(\overline{X},\overline{Y})\,\overline{Z},\overline{W})\\
    &+y\,z\,\g(\R(\overline{X},E_3)\,E_3,\overline{W})
    +x\,z\,\g(\R(E_3,\overline{Y})\,E_3,\overline{W})\\
   &+w\,x\,\g(\R(E_3,\overline{Y})\,\overline{Z},E_3)
    +y\,w\,\g(\R(\overline{X},E_3)\,\overline{Z},E_3),
    \end{align*}
    where it's easy to see that
    $$
    \g(\R(\overline{X},\overline{Y})\,\overline{Z},\overline{W})=3\tau^2\,[\g(\overline{X},\overline{W})\,\g(\overline{Y},\overline{Z})-\g(\overline{X},\overline{Z})\,\g(\overline{Y},\overline{W})].$$
    Also, as $$\R(\overline{X},E_3)\,E_3=\tau^2\,\overline{X},\qquad 
    \R(E_3,\overline{Y})\,E_3=-\tau^2\,\overline{Y},$$
    it results that 
     \begin{align*}
    \g(\R(X,Y)\,Z,W)&=3\tau^2\,[\g(\overline{X},\overline{W})\,\g(\overline{Y},\overline{Z})-\g(\overline{X},\overline{Z})\,\g(\overline{Y},\overline{W})]\\
    &+\tau^2\,[z\,y\,\g(\overline{X},\overline{W})+x\,w\,\g(\overline{Y},\overline{Z})-x\,z\,\g(\overline{Y},\overline{W})-w\,y\,\g(\overline{X},\overline{Z})]\\
    &=3\tau^2\,[\g(Y,Z)\g(X,W)-\g(X,Z)\,\g(Y,W)]\\
&+4\tau^2\,[\g(Y,E_3)\,\g(Z,E_3)\,\g(X,W)-\g(X,E_3)\,\g(Z,E_3)\,\g(Y,W)\\
&+\g(Y,Z)\,\g(X,E_3)\,\g(E_3,W)-\g(X,Z)\,\g(Y,E_3)\,\g(E_3,W)].
    \end{align*}
   Since $W$ is arbitrary, we obtain the equation~\eqref{eq-tensor-R}.
    \end{proof}
   
\section{The structure equations for surfaces in $\H$}\label{tre}
In this section, we will determine the structure equations for a surface ${\mathcal M}$ immersed into the Lorentzian Heisenberg group $\H$ that will be used, in the following sections, to study the constant angle surfaces in this ambient space. We remember that a surface ${\mathcal M}$ is called {\it spacelike} (respectively, {\it timelike}) if  the induced metric on ${\mathcal M}$ via the immersion is a Riemannian (respectively, Lorentzian) metric.

Let ${\mathcal M}$ be an oriented pseudo-Riemannian surface immersed into $\H$ and let  $N$ be a unit normal vector field, that is $g_\tau(N,N)=\varepsilon$, where $\varepsilon=-1$ (respectively,  $\varepsilon=1$) if ${\mathcal M}$ is a spacelike (respectively, a timelike) surface.

The Gauss and Weingarten formulas, for all $X,Y\in C(T{\mathcal M})$, are
\begin{equation}\label{gauss-wein}\begin{aligned}
    \n_X Y&=\nabla_X Y+\alpha(X,Y),\\
    \n_X N&=-A(X),
    \end{aligned}
\end{equation}
where with $A$ we have indicated  the shape operator of ${\mathcal M}$ in $\H$, with $\nabla$ the induced Levi-Civita connection on ${\mathcal M}$ and by $\alpha$  the second fundamental form of ${\mathcal M}$ in $\H$. Note that the second fundamental form can be written as
$$\alpha(X, Y) =\varepsilon\,\g(A(X),Y)\,N,\qquad X,Y\in C(T{\mathcal M}).$$
If we project the vector field $E_3$ onto the tangent plane to ${\mathcal M}$, we have $$E_3=T+\nu\, N,$$ for a certain smooth function $\nu$ defined on ${\mathcal M}$. Here $T$ is the tangent part of $E_3$ which satisfies 
\begin{equation}\label{t}
g_\tau(T,T)=-(1+\varepsilon\,\nu^2).
\end{equation}

We observe that, for all $X\in T{\mathcal M}$, we have that
\begin{equation}\label{eq1}\begin{aligned}
    \n_X E_3&=\n_X T+X(\nu)\,N+\nu\,\n_X N\\&=\nabla_X T+X(\nu)\,N+\varepsilon\,g_\tau(A(X),T)\,N-\nu\,A(X).\end{aligned}
\end{equation}
On the other hand, writing $X=\sum_{i=1}^3 X_i \,E_i$ and using \eqref{principal}, we get:
\begin{equation}\label{eq2}\begin{aligned}
    \n_X E_3&=
    \tau\, X\wedge E_3\\&=\varepsilon\,\tau\,g_\tau(JX,T)\,N-\tau\,\nu JX,
    \end{aligned}
\end{equation}
where $JX=N\wedge X$ denotes the rotation of angle $\pi/2$ on $T\mathcal{M}$ and it satisfies
\begin{equation}\label{j}
\g(JX,JX)=-\varepsilon\,\g(X,X),\qquad J^2X=\varepsilon\,X.
\end{equation}
Identifying the tangent and normal components of \eqref{eq1} and \eqref{eq2} respectively, we obtain
\begin{equation}\label{eq3}
    \nabla_X T= \nu\,(A(X)-\tau\,JX)
\end{equation}
and
\begin{equation}\label{eq4}
   X(\nu)=-\varepsilon\, g_\tau(A(X)-\tau\,JX,T).
\end{equation}
In the following result, we will give the expression of the Gauss and Codazzi equations for a pseudo-Riemannian surface ${\mathcal M}$ in $\H$.

 \begin{proposition}
    Under the previous conditions, the Gauss and Codazzi equations in $\H$ are given, respectively, by:
    \begin{equation}\label{eq-gauss}
K=\overline{K}+\varepsilon\,\mathrm{det} A=\varepsilon\,(\mathrm{det} A-4\,\tau^2\,\nu^2)-\tau^2
    \end{equation}
and
    \begin{equation}\label{eq-codazzi}
\nabla_X A(Y)-\nabla_Y A(X)-A[X,Y]=-4\,\varepsilon\,\nu\,\tau^2\,[g_\tau(Y,T)X-g_\tau(X,T)Y],
    \end{equation}
    where $X$ and $Y$ are tangent vector fields on ${\mathcal M}$, $K$ is the Gauss curvature of ${\mathcal M}$ and $\overline{K}$ denotes the sectional curvature in $\H$ of the plane tangent to ${\mathcal M}$. 
    \end{proposition}
    \begin{proof}
    We start proving the equation~\eqref{eq-gauss}. Using that 
    $$ \begin{aligned}&\g(\alpha(X,X),\alpha(Y,Y))-\g(\alpha(X,Y),\alpha(X,Y))^2\\&=\varepsilon\,[\g(A(X),X)\,\g(A(Y),Y)-\g(A(X),Y)^2],
    \end{aligned}$$ the Gauss equation can be written as
     \begin{equation}\label{gauss1}
     K=\overline{K}+\varepsilon\,\frac{\g(A(X),X)\g(A(Y),Y)-\g(A(X),Y)^2}{\g(X,X)\g(Y,Y)-\g(X,Y)^2}.
    \end{equation}
    If we suppose that $\{X,Y\}$ is a local orthonormal frame on ${\mathcal M}$, i.e. $\g(X,X)=1$, $\g(X,Y)=~0$, $\g(Y,Y)=-\varepsilon$, we get
     $$\g(A(X),X)\,\g(A(Y),Y)-\g(A(X),Y)^2=-\varepsilon\,\mathrm{det} A.$$
     Also, from \eqref{eq-tensor-R}, we obtain that 
    \begin{equation}\label{gauss2}
\begin{aligned}-\varepsilon\,\overline{K}=\g(\R(X,Y)Y,X)&=-3\varepsilon\,\tau^2+4\tau^2\,[\g(Y,T)^2-\varepsilon\,\g(X,T)^2]\\
    &=-\varepsilon\,\tau^2\,[3+4\,\g(T,T)].
    \end{aligned}
    \end{equation}
    Now, substituting \eqref{gauss2} in \eqref{gauss1}, and using \eqref{t}, we have the equation~\eqref{eq-gauss}.
    
    To obtain \eqref{eq-codazzi}, we start from the Codazzi equation for hypersurfaces that is given by:
    $$\g(\R(X,Y)Z,N)=\g(\nabla_X A(Y)-\nabla_Y A(X)-A[X,Y],Z).$$
    Also, from Proposition~\ref{tensor-R} we get
    $$\begin{aligned}\R(X,Y)N&=4\tau^2\,\g(N,E_3)\,[g_\tau(Y,E_3)X-g_\tau(X,E_3)Y]\\
    &=4\,\varepsilon\,\nu\,\tau^2\,[g_\tau(Y,T)X-g_\tau(X,T)Y].
    \end{aligned}$$ Therefore, we obtain \eqref{eq-codazzi}.
    \end{proof}
  
Now, we are ready to begin the study of the constant angle surfaces in $\h_3(\tau)$. Firstly, we give the following:
\begin{definition}
Let ${\mathcal M}$ be an oriented pseudo-Riemannian surface in the Lorentzian Heisenberg group $\H$ and let $N$ be a unit normal vector field, with $\g(N,N)=\varepsilon$. We say that ${\mathcal M}$  is a {\it helix surface} or a {\it constant angle surface} if the function $\nu:=\varepsilon\,\g(N,E_3)$ is constant at every point of the surface.
\end{definition}

\section{Constant angle spacelike surfaces}\label{spacelike}
Let ${\mathcal M}$ be a spacelike surface  in $\h_3(\tau)$.  As $\varepsilon=-1$, from the equation~\eqref{t} it follows that (up to the orientation of $N$) we can write $\nu=\cosh\vartheta$, where $\vartheta\geq 0$ is called {\it the hyperbolic angle function} between $N$ and $E_3$. 

From now on, we assume that the 
function $\vartheta$ is constant.
Note that $\vartheta\neq 0$. In fact, if it were then the vector fields $E_1$ and $E_2$   would be tangent to the surface ${\mathcal M}$, which is absurd since  as $\tau\neq 0$ the horizontal distribution of the submersion $\pi$ is not integrable. 
\begin{lemma}\label{princ}
Let ${\mathcal M}$ be a  helix spacelike surface in $\H$ with constant angle $\vartheta>0$. Then, we have the following properties.
\begin{itemize}
  \item[(i)] With respect to the basis $\{T,JT\}$, the matrix associated to the shape operator $A$ takes the form
$$
A=\left(
\begin{array}{cc}
0 & -\tau \\
-\tau & \lambda\\
 \end{array}
 \right),
  $$
  for some smooth function $\lambda$ on ${\mathcal M}$.
  \item[(ii)] The Levi-Civita connection $\nabla$ of ${\mathcal M}$ is given by
  $$\nabla_T T=-2\tau\cosh\vartheta\, JT,\qquad \nabla_{JT} T=\lambda\cosh\vartheta\, JT,$$
  $$\nabla_T JT=2\tau\cosh\vartheta\, T,\qquad \nabla_{JT} JT=-\lambda\cosh\vartheta\, T.$$
  \item[(iii)] The Gauss curvature of ${\mathcal M}$ is constant and satisfies $$K=4\tau^2\,\cosh^2\vartheta.$$
  \item[(iv)] The function $\lambda$ satisfies the equation
  \begin{equation}\label{lambda}
    T (\lambda)+\lambda^2\,\cosh\vartheta+4\tau^2\,\cosh^3\vartheta=0.
  \end{equation}
\end{itemize}
\end{lemma}
\begin{proof} Point (i) follows directly from \eqref{eq4}.
From \eqref{eq3} and using $$g_\tau(T,T)=g_\tau(JT,JT)=\sinh^2\vartheta,\qquad g_\tau(T,JT)=0,$$ we obtain (ii). From the Gauss equation~\eqref{eq-gauss} in $\H$, we have that the Gauss curvature of ${\mathcal M}$ is given by
$$K=4\tau^2\,\nu^2-(\mathrm{det} A+\tau^2)=4\tau^2\,\cosh^2\vartheta.$$
 Finally,  equation~\eqref{lambda} follows from the Codazzi equation~\eqref{eq-codazzi} putting $X=T$, $Y=JT$ and using (ii).
 In fact, it is easy to see that
 $$4\tau^2\,\cosh\vartheta\,[g_\tau(JT,T)T-g_\tau(T,T)JT]=-4\tau^2\,\cosh\vartheta\,\sinh^2\vartheta\, JT$$ and
 $$\begin{aligned}&\nabla_T A(JT)-\nabla_{JT} A(T)-A[T,JT]\\&=
 \nabla_T (-\tau\, T+\lambda\, JT)-\nabla_{JT} (-\tau\, JT)-A(2\tau \cosh\vartheta\,T-\lambda \cosh\vartheta\, JT)\\&=
 (4\tau^2\cosh\vartheta+T(\lambda)+\lambda^2\,\cosh\vartheta)\,JT.
 \end{aligned}$$
\end{proof}

From $g_\tau (E_3,N)=-\cosh\vartheta$, it follows that there exists a smooth function $\varphi$ on ${\mathcal M}$ such that
$$N=\sinh\vartheta\cos\varphi\,E_1+\sinh\vartheta\sin\varphi\,E_2+\cosh\vartheta E_3.$$
Therefore, we can write
\begin{equation}\label{eq:def-T}T=E_3-\cosh\vartheta\,N=-\sinh\vartheta\,[\cosh\vartheta\cos\varphi\,E_1+\cosh\vartheta\sin\varphi\,E_2+\sinh\vartheta\,E_3]
\end{equation}
and \begin{equation}\label{eq:def-JT}JT=\sinh\vartheta\,(\sin\varphi\,E_1-\cos\varphi\,E_2).
\end{equation}
Moreover, we have
\begin{equation}
\left\{\begin{aligned}\label{eqTJ}
    A(T)&=-\n_T N=[T(\varphi)+\tau\cosh^2\vartheta+\tau\sinh^2\vartheta]\,JT,\\
     A(JT)&=-\n_{JT} N=(JT)(\varphi)\,JT-\tau\,T.
\end{aligned}\right.
\end{equation}
Comparing \eqref{eqTJ} with (i) of Lemma~\ref{princ}, it results that
\begin{equation}\label{eqTJ1}\left\{\begin{aligned}
(JT)(\varphi)&=\lambda,\\
T(\varphi)&=-2\tau\,\cosh^2\vartheta.
\end{aligned}
\right.
\end{equation}
Also, as $$[T,JT]=\cosh\vartheta\,(2\tau\, T-\lambda\,JT),$$ the compatibility condition of system~\eqref{eqTJ1}:
$$(\nabla_T JT-\nabla_{JT} T)(\varphi)=[T,JT](\varphi)=T (JT(\varphi))-JT (T(\varphi))$$ is equivalent to
\eqref{lambda}.

We now choose local coordinates $(u,v)$ on ${\mathcal M}$ such that
\begin{equation}\label{eq:local-coordinates}
\partial_u=T.
\end{equation}
 Also, as $\partial_v$ is tangent to ${\mathcal M}$, it can be written in the form
 \begin{equation}\label{eq-definition-Fv}
 \partial_v=a\,T+b\,JT,
 \end{equation}
  for certain functions $a=a(u,v)$ and $b=b(u,v)$. As
$$0=[\partial_u,\partial_v]=(a_u+2\tau\, b\,\cosh\vartheta)\,T+(b_u-b\,\lambda\,\cosh\vartheta)\,JT,$$ then
\begin{equation}\label{eqab}\left\{\begin{aligned}
a_u&=-2\tau\, b\,\cosh\vartheta,\\
b_u&=b\,\lambda\,\cosh\vartheta.
\end{aligned}
\right.
\end{equation}
Moreover, the equation \eqref{lambda} of Lemma~\ref{princ} can be written as
\begin{equation}\label{eq-lambda-du}
\lambda_u+\cosh\vartheta\,\lambda^2+4\tau^2\,\cosh^3\vartheta=0.
\end{equation}

Integrating \eqref{eq-lambda-du}, we obtain that
\begin{equation}
    \lambda(u,v)=2\tau\,\cosh\vartheta\,\tan [\eta(v)-2\tau(\cosh\vartheta)^2\,u],
\end{equation}
for some smooth function $\eta$ depending on $v$ and  we can solve  system~\eqref{eqab}. Remark that we are interested in only one coordinate system on the surface ${\mathcal M}$ and, hence, we only need one solution for $a$ and $b$, for example:
$$\left\{\begin{aligned}
a(u,v)&=\frac{\sin (\eta(v)-2\tau(\cosh\vartheta)^2\,u)}{\cosh\vartheta},\\
b(u,v)&=\cos(\eta(v)-2\tau(\cosh\vartheta)^2\,u).
\end{aligned}
\right.
$$
Moreover, using these expressions, we have that the system~\eqref{eqTJ1} becomes
\begin{equation}\label{eqTJ2}\left\{\begin{aligned}\nonumber
\varphi_u&=-2\tau \cosh^2\vartheta,\\
\varphi_v&=0,
\end{aligned}
\right.
\end{equation}
of which the general solution is given by
\begin{equation}\label{eqphi2}
    \varphi(u,v)=-2\tau \cosh^2\vartheta\,u+c,
\end{equation}
where $c$ is a real constant.

With respect to the local coordinates $(u,v)$ chosen above, we have the following characterization of the position vector of a  helix spacelike surface.

\begin{theorem}\label{teo-one}
Let ${\mathcal M}$ be a  helix spacelike surface in $\H$ with constant angle $\vartheta>0$.  Then, with respect to the local coordinates $(u,v)$ on ${\mathcal M}$ defined in \eqref{eq:local-coordinates} and \eqref{eq-definition-Fv},  the position vector $F$ of ${\mathcal M}$ in $\r^3$ is given by
\begin{equation}\label{para}
\begin{aligned}F(u,v)=\Big(&\frac{\tanh\vartheta}{2\tau}\sin u+f_1(v),-\frac{\tanh\vartheta}{2\tau}\cos u+f_2(v),\\&-\frac{ (\sinh \vartheta)^2 }{2}u+\frac{\tanh \vartheta}{2}\,[f_1(v)\cos u+f_2(v)\sin u]+f_3(v)\Big),
\end{aligned}
\end{equation}
with $$f_1'(v)^2+f_2'(v)^2=(\sinh\vartheta)^2,\qquad  f_3'(v)=\tau\,(f_2(v)\,f_1'(v)-f_1(v)\,f_2'(v)).$$
\end{theorem}
\begin{proof}
Let ${\mathcal M}$ be a  helix spacelike surface in $\H$ with constant angle $\vartheta\in (0,+\infty)$ and let $F$ be the position vector of ${\mathcal M}$ in $\r^3$. Then, with respect to the local coordinates $(u,v)$ on ${\mathcal M}$ defined in \eqref{eq:local-coordinates} and \eqref{eq-definition-Fv}, we can write $F(u,v)=(F_1(u,v),F_2(u,v),F_3(u,v))$, with $(u,v)\in\Omega\subset\r^2$. By definition, taking into account \eqref{eq:def-T} and \eqref{eq:def-JT}, we have that
$$
\begin{aligned}\partial_u F&=(\partial_uF_1,\partial_uF_2,\partial_uF_3)=T\\
&=-\sinh\vartheta\,[\cosh\vartheta\cos\varphi\,{E_1}_{|F(u,v)}+\cosh\vartheta\sin\varphi\,{E_2}_{|F(u,v)}+\sinh\vartheta\,{E_3}_{|F(u,v)}]\,
\end{aligned}
$$
and
$$
\begin{aligned}\partial_v F&=(\partial_vF_1,\partial_vF_2,\partial_vF_3)=a\,T+b\,JT\\
&=\sinh\vartheta\,[(-a\cosh\vartheta\cos\varphi+b\sin\varphi)\,{E_1}_{|F(u,v)}\\&-(a\cosh\vartheta\sin\varphi+b\cos\varphi)\,{E_2}_{|F(u,v)}-a\sinh\vartheta\,{E_3}_{|F(u,v)}]\,.
\end{aligned}
$$
Therefore, using the expression of $E_1$, $E_2$ and $E_3$ with respect to the coordinates vector fields of $\r^3$, it results that
\begin{equation}\label{eqprime}\left\{\begin{aligned}
\partial_uF_1&=-\sinh\vartheta\cosh\vartheta\cos\varphi,\\
\partial_uF_2&=-\sinh\vartheta\cosh\vartheta\sin\varphi,\\
\partial_uF_3&=-\sinh\vartheta\,(\tau\cosh\vartheta\cos\varphi\,F_2-\tau\cosh\vartheta\sin\varphi\,F_1+ \sinh\vartheta)
\end{aligned}
\right.
\end{equation}
and
\begin{equation}\label{eqsecond}\left\{\begin{aligned}
\partial_vF_1&=\sinh\vartheta\,(-a\cosh\vartheta\cos\varphi+b\sin\varphi),\\
\partial_vF_2&=-\sinh\vartheta\,(a\cosh\vartheta\sin\varphi+b\cos\varphi),\\
\partial_vF_3&=\sinh\vartheta\,[\tau(-a\cosh\vartheta\cos\varphi+b\sin\varphi)\,F_2\\&+
\tau\,(a\cosh\vartheta\sin\varphi+b\cos\varphi)\,F_1-a \sinh\vartheta]\,.
\end{aligned}
\right.
\end{equation}
From the first two equations of \eqref{eqprime}, we obtain that
$$\left\{\begin{aligned}
F_1(u,v)&=\frac{\tanh\vartheta}{2\tau}\sin\varphi(u)+f_1(v),\\
F_2(u,v)&=-\frac{\tanh\vartheta}{2\tau}\cos\varphi(u)+f_2(v).
\end{aligned}
\right.$$
Then, using these expressions in the third equation of \eqref{eqprime} and integrating, we get
$$F_3(u,v)=-\frac{ (\sinh\vartheta)^2\,u }{2}+\frac{\tanh\vartheta}{2}\,(f_1(v)\cos\varphi (u)+f_2(v)\sin\varphi (u))+f_3(v).$$
Consequently, from \eqref{eqsecond}, we have that
\begin{equation}\label{fi-spazio}
\left\{\begin{aligned}
f_1'(v)&=-\sinh\vartheta\,\sin(\eta(v)-c),\\
f_2'(v)&=-\sinh\vartheta\,\cos(\eta(v)-c),\\
f_3'(v)&=\tau\,(f_2(v)\,f_1'(v)-f_1(v)\,f_2'(v)).
\end{aligned}\right.
\end{equation}
Using the change of variable $\varphi(u)\mapsto u$, we obtain the parametrization given in \eqref{para}.
\end{proof}

Now, we present some examples of constant angle spacelike surfaces in $\h_3(\tau)$ obtained using the parametrization given in the Theorem~\ref{teo-one}.

\begin{example}
Choosing $\eta(v)=v+c$ in \eqref{fi-spazio}, we get
$$
\left\{\begin{aligned}
f_1(v)&=\sinh\vartheta\,\cos v,\\
f_2(v)&=-\sinh\vartheta\,\sin v,\\
f_3(v)&=\tau\,v\,\sinh^2\vartheta.
\end{aligned}\right.
$$
Substituting these expressions in \eqref{para} we have explicit  parametrizations of helix spacelike surfaces that depend only of the hyperbolic angle $\vartheta$.
\begin{figure}[!h]
\begin{center}
\vspace{-.3 cm}
\includegraphics[width=0.12\linewidth]{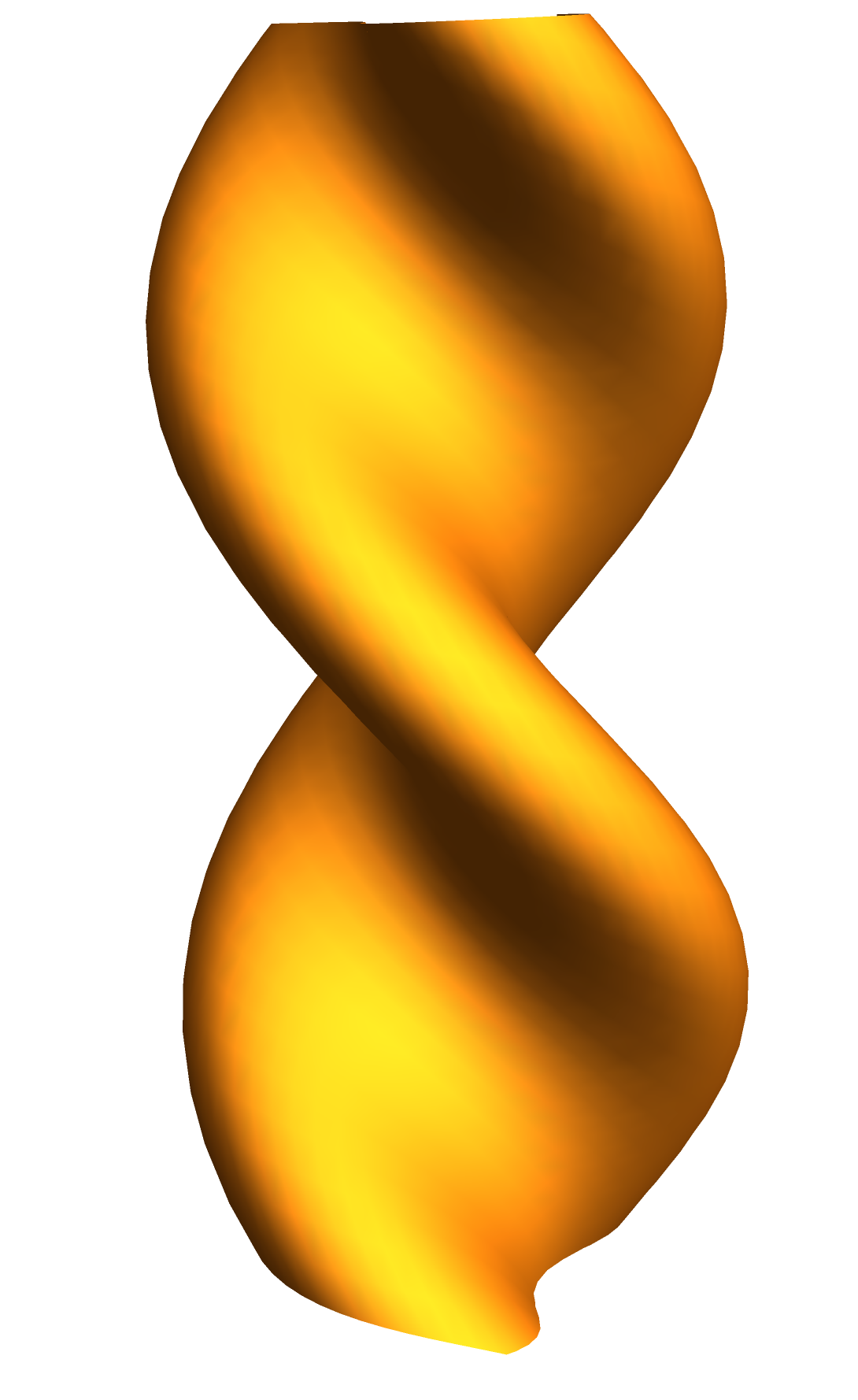} \qquad 
\includegraphics[width=0.11\linewidth]{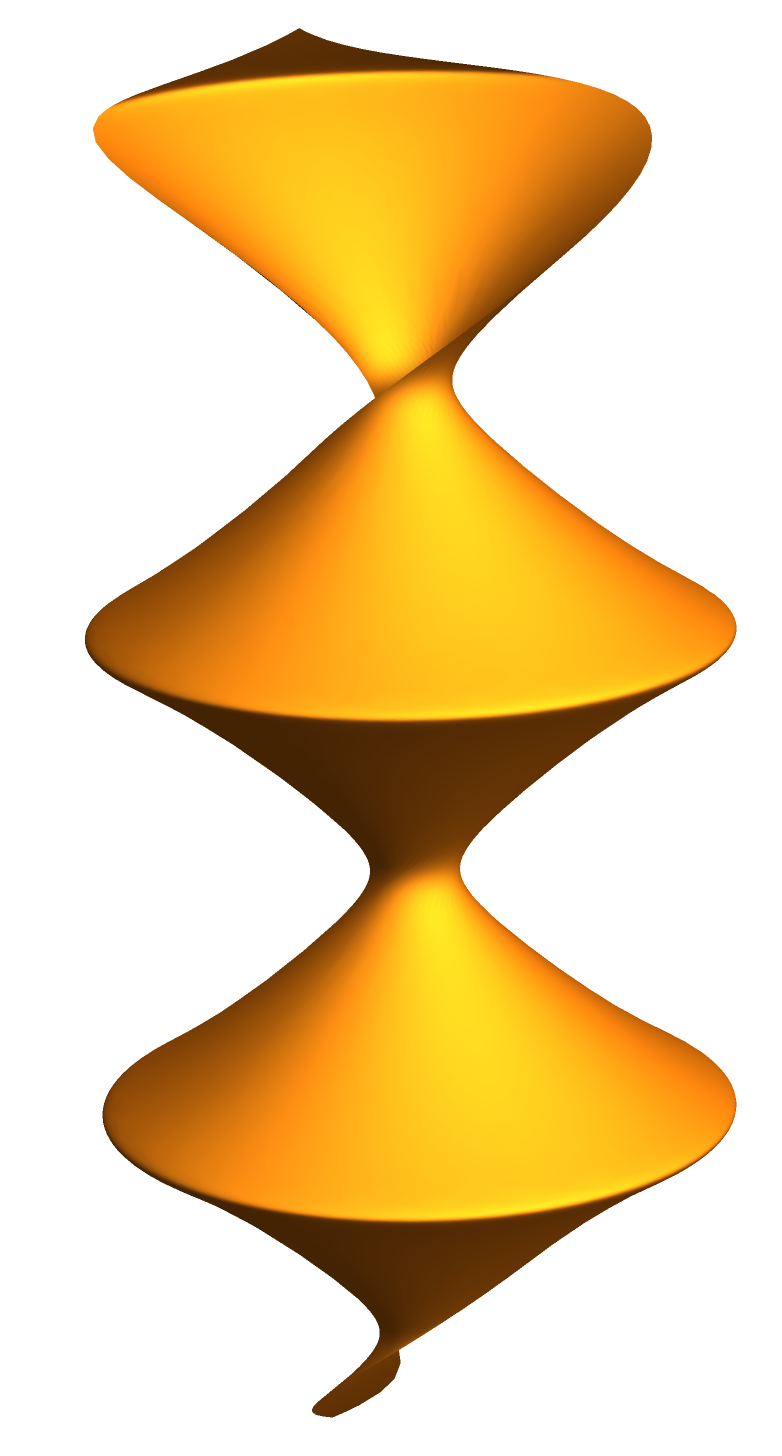}\qquad 
\includegraphics[width=0.13\linewidth]{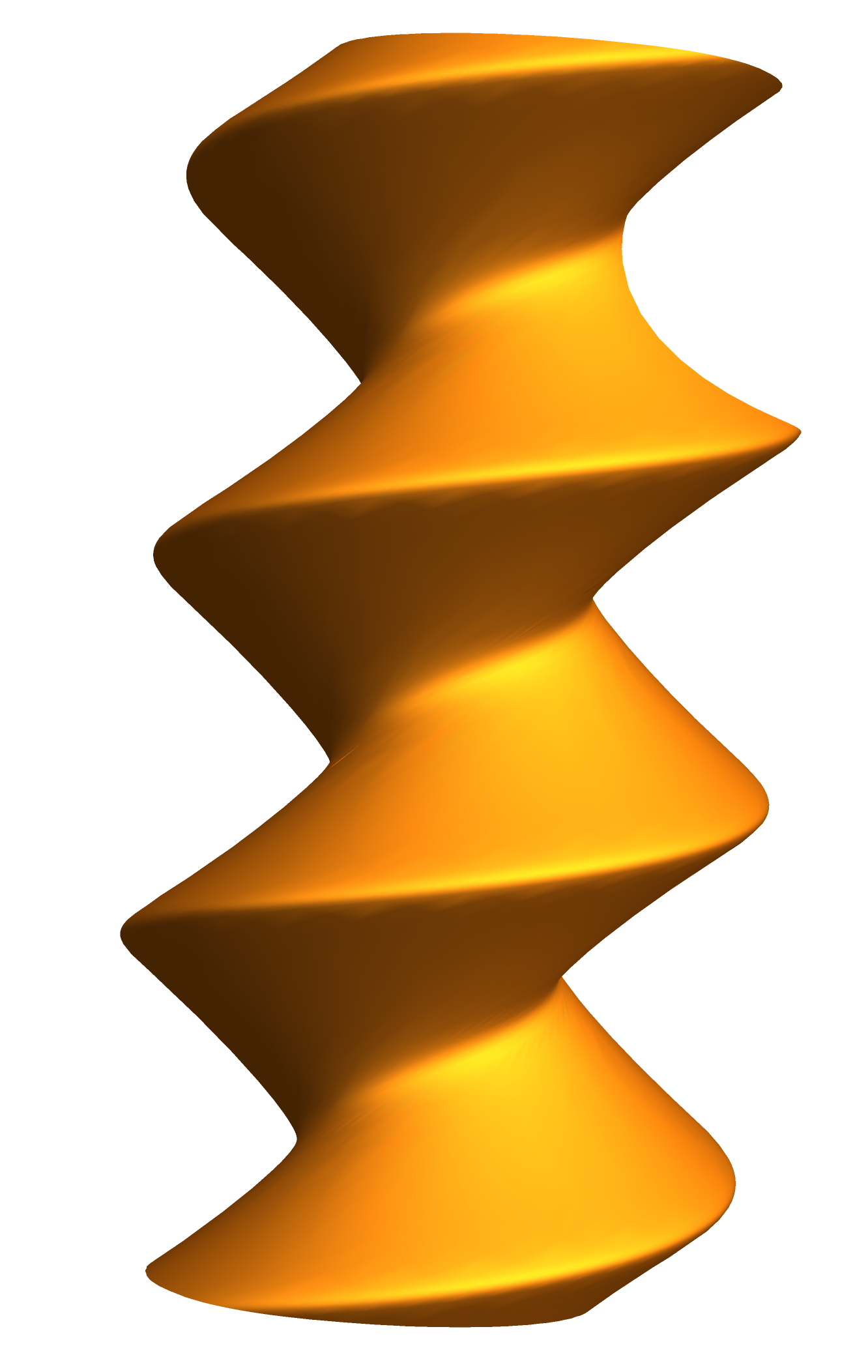}\qquad 
\includegraphics[width=0.22\linewidth]{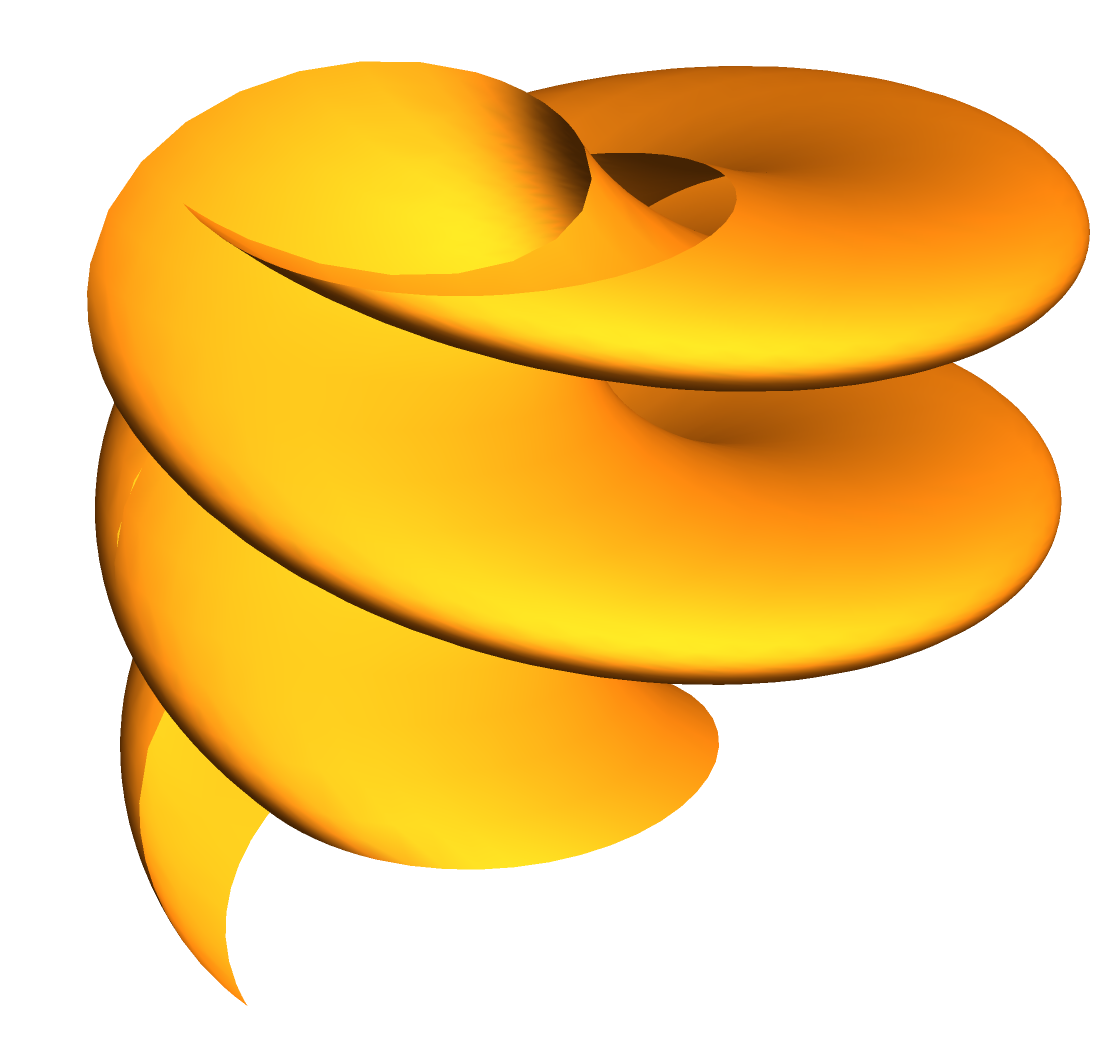}
\end{center}
\vspace{-.2 cm}
\caption{\it  Constant angle spacelike surfaces for $\vartheta =\pi/3$, $\vartheta =\pi/4$, $\vartheta =\pi/6$  and $\vartheta =\pi/8$.}	
\label{Figure label}
\end{figure}
\end{example}

\section{Constant angle timelike surfaces}\label{timelike}
Now we are going to study, following the same procedure as in  Section~\ref{spacelike}, the constant angle timelike surfaces in $\H$.

Let ${\mathcal M}$ be a timelike surface in the Lorentzian Heisenberg group $\H$.  As $\varepsilon=1$, from the equation~\eqref{t} it follows that (up to the orientation of $N$) we can write $\nu=\sinh\vartheta$, where $\vartheta$ is the {\it hyperbolic angle function} between $N$ and $E_3$. We observe that, from \eqref{j}, it results that
$$\g(JT,JT)=-\g(T,T)=\cosh^2\vartheta.$$

We suppose that the function $\vartheta$ is constant and we observe that if $\vartheta=0$, we have that $E_3$ is always tangent to
${\mathcal M}$ and, therefore, ${\mathcal M}$ is a Hopf cylinder. Therefore, from now on we assume that the constant angle $\vartheta\neq 0$.

 \begin{lemma}\label{princ1}
Let ${\mathcal M}$ be a  helix timelike surface in $\H$ with constant angle $\vartheta\neq 0$. Then, we have the following properties.
\begin{itemize}
  \item[(i)] With respect to the basis $\{T,JT\}$, the matrix associated to the shape operator $A$ takes the form
$$
A=\left(
\begin{array}{cc}
0 & \tau \\
-\tau & \lambda\\
 \end{array}
 \right),
  $$
  for some smooth function $\lambda$ on ${\mathcal M}$.
  \item[(ii)] The Levi-Civita connection $\nabla$ of ${\mathcal M}$ is given by
  $$\nabla_T T=-2\tau\,\sinh\vartheta\,JT,\qquad \nabla_{JT} T=\lambda\sinh\vartheta\, JT,$$
  $$\nabla_T JT=-2\tau\,\sinh\vartheta\,T,\qquad \nabla_{JT} JT=\lambda\sinh\vartheta\, T.$$
  \item[(iii)] The Gauss curvature of ${\mathcal M}$ is the constant given by: $$K=-4\tau^2\,\sinh^2\vartheta.$$
  \item[(iv)] The function $\lambda$ satisfies the equation
   \begin{equation}\label{lambda1}
    T (\lambda)+\lambda^2\,\sinh\vartheta+4\tau^2\,\sinh^3\vartheta=0.
  \end{equation}
\end{itemize}
\end{lemma}
\begin{proof}
The proof is analogous to that of Lemma~\ref{princ}, taking into account that $T$ is timelike, $JT$ is spacelike and that the operator of rotation $J$ satisfies $J^2=\textrm{I}$ (see \eqref{j}).
\end{proof}
From $g_\tau (E_3,N)=\sinh\vartheta$, it follows that there exists a smooth function $\varphi$ on ${\mathcal M}$ such that
$$N=\cosh\vartheta\cos\varphi\,E_1+\cosh\vartheta\sin\varphi\,E_2-\sinh\vartheta\, E_3.$$
Consequently, we obtain that
\begin{equation}\label{eq:def-T1}
\left\{\begin{aligned}
T&=E_3-\sinh\vartheta\,N=\cosh\vartheta\,[-\sinh\vartheta\cos\varphi\,E_1-\sinh\vartheta\sin\varphi\,E_2+
\cosh\vartheta\,E_3],\\
JT&=\cosh\vartheta\,(\sin\varphi\,E_1-\cos\varphi\,E_2).
\end{aligned}
\right.
\end{equation}

In this case, we have that
\begin{equation}\begin{aligned}\label{eqTJb}
    A(T)&=-\n_T N=[T(\varphi)-\tau\cosh^2\vartheta-\tau\sinh^2\vartheta]\,JT,\\
     A(JT)&=-\n_{JT} N=(JT)(\varphi)\,JT+\tau\,T.
\end{aligned}\end{equation}
Comparing \eqref{eqTJb} with (i) of Lemma~\ref{princ1}, we get
\begin{equation}\label{eqTJbi}\left\{\begin{aligned}
(JT)(\varphi)&=\lambda,\\
T(\varphi)&=2\tau\,\sinh^2\vartheta.
\end{aligned}
\right.
\end{equation}
Also, as $$[T,JT]=-\sinh\vartheta\,(2\tau\, T+\lambda\,JT),$$ the compatibility condition of system~\eqref{eqTJbi} is given by:
$$(\nabla_T JT-\nabla_{JT} T)(\varphi)=[T,JT](\varphi)=T (JT(\varphi))-JT (T(\varphi))$$ and  it is equivalent to the equation~\eqref{lambda1}.

Now, we choose local coordinates $(u,v)$ on ${\mathcal M}$ such that
\begin{equation}\label{eq:local-coordinates-b}
\partial_u=T,\qquad
 \partial_v=a\,T+b\,JT\,,
 \end{equation}
  for certain functions $a=a(u,v)$ and $b=b(u,v)$. As
$$0=[\partial_u,\partial_v]=(a_u-2\tau\, b\,\sinh\vartheta)\,T+(b_u-b\,\lambda\,\sinh\vartheta)\,JT,$$ then
\begin{equation}\label{eqab-b}\left\{\begin{aligned}
a_u&=2\tau\, b\,\sinh\vartheta,\\
b_u&=b\,\lambda\,\sinh\vartheta.
\end{aligned}
\right.
\end{equation}
Moreover, the equation \eqref{lambda1} of Lemma~\ref{princ1} can be written as
$$
\lambda_u+\sinh\vartheta\,\lambda^2+4\tau^2\,\sinh^3\vartheta=0
$$
and, solving this equation, one finds
$$
    \lambda(u,v)=2\tau\,\sinh\vartheta\,\tan [\eta(v)-2\tau(\sinh\vartheta)^2\,u],
$$
for some smooth function $\eta$ depending on $v$. As we are interested in only one coordinate system on the surface ${\mathcal M}$,  we can consider the following solution of the system~\eqref{eqab-b}:
$$\left\{\begin{aligned}
a(u,v)&=-\frac{\sin (\eta(v)-2\tau(\sinh\vartheta)^2\,u)}{\sinh\vartheta},\\
b(u,v)&=\cos(\eta(v)-2\tau(\sinh\vartheta)^2\,u).
\end{aligned}
\right.
$$
Also, using these expressions, we have that the general solution of the system~\eqref{eqTJbi}  is given by:
$$
    \varphi(u,v)=2\tau (\sinh\vartheta)^2\,u+c, \qquad c\in\r.
$$

\begin{theorem}\label{teo-two}
Let ${\mathcal M}$ be a helix timelike surface in $\H$ with constant angle $\vartheta\neq 0$.  Then, with respect to the local coordinates $(u,v)$ on ${\mathcal M}$ defined in \eqref{eq:local-coordinates-b}  the position vector $F$ of ${\mathcal M}$ in $\r^3$ is given by:
\begin{equation}\label{para-b}
\begin{aligned}F(u,v)=\Big(&-\frac{\coth\vartheta}{2\tau}\sin u+f_1(v),\frac{\coth\vartheta}{2\tau}\cos u+f_2(v),\\&\frac{ (\cosh \vartheta)^2 }{2}u-\frac{\coth \vartheta}{2}\,[f_1(v)\cos u+f_2(v)\sin u]+f_3(v)\Big),
\end{aligned}
\end{equation}
with $$f_1'(v)^2+f_2'(v)^2=(\cosh\vartheta)^2,\qquad  f_3'(v)=\tau\,(f_2(v)\,f_1'(v)-f_1(v)\,f_2'(v)).$$
\end{theorem}
\begin{proof}
With respect to the local coordinates $(u,v)$ on the  helix timelike surface ${\mathcal M}$, given in \eqref{eq:local-coordinates-b}, we can parametrize the surface as $$F(u,v)=(F_1(u,v),F_2(u,v),F_3(u,v)),\quad (u,v)\in\Omega\subset\r^2.$$ From the expressions \eqref{eq:def-T1}, it results that
\begin{equation}\label{eqprime1}\left\{\begin{aligned}
\partial_uF_1&=-\sinh\vartheta\cosh\vartheta\cos\varphi,\\
\partial_uF_2&=-\sinh\vartheta\cosh\vartheta\sin\varphi,\\
\partial_uF_3&=\cosh\vartheta\,(-\tau\sinh\vartheta\cos\varphi\,F_2+\tau\sinh\vartheta\sin\varphi\,F_1+ \cosh\vartheta)
\end{aligned}
\right.
\end{equation}
and
\begin{equation}\label{eqsecond1}\left\{\begin{aligned}
\partial_vF_1&=\cosh\vartheta\,(-a\sinh\vartheta\cos\varphi+b\sin\varphi),\\
\partial_vF_2&=-\cosh\vartheta\,(a\sinh\vartheta\sin\varphi+b\cos\varphi),\\
\partial_vF_3&=\cosh\vartheta\,[\tau(-a\sinh\vartheta\cos\varphi+b\sin\varphi)\,F_2\\&+
\tau\,(a\sinh\vartheta\sin\varphi+b\cos\varphi)\,F_1+a \cosh\vartheta].
\end{aligned}
\right.
\end{equation}
Integrating \eqref{eqprime1}, we obtain that
$$\left\{\begin{aligned}
F_1(u,v)&=-\frac{\coth\vartheta}{2\tau}\sin\varphi(u)+f_1(v),\\
F_2(u,v)&=\frac{\coth\vartheta}{2\tau}\cos\varphi(u)+f_2(v),\\
F_3(u,v)&=\frac{ (\cosh\vartheta)^2\,u }{2}-\frac{\coth\vartheta}{2}\,(f_1(v)\cos\varphi (u)+f_2(v)\sin\varphi (u))+f_3(v),
\end{aligned}
\right.$$
where, from \eqref{eqsecond1}, the functions $f_i(v)$, $i=1,2,3,$ satisfy 
the following relations:
\begin{equation}\label{fi-tempo}\left\{\begin{aligned}
f_1'(v)&=\cosh\vartheta\,\sin(\eta(v)+c),\\
f_2'(v)&=-\cosh\vartheta\,\cos(\eta(v)+c),\\
f_3'(v)&=\tau\,(f_2(v)\,f_1'(v)-f_1(v)\,f_2'(v)).
\end{aligned}\right.
\end{equation}
Finally, using the change of variable $\varphi(u)\mapsto u$, we obtain the parametrization of ${\mathcal M}$ given in \eqref{para-b}.
\end{proof}

We end the section giving some examples of constant angle timelike surfaces in $\h_3(\tau)$ constructed from the parametrization  obtained in the Theorem~\ref{teo-two}.

\begin{example}
If we choose $\eta(v)=v-c$ in \eqref{fi-tempo}, we have the expressions:
$$
\left\{\begin{aligned}
f_1(v)&=-\cosh\vartheta\,\cos v,\\
f_2(v)&=-\cosh\vartheta\,\sin v,\\
f_3(v)&=-\tau\,v\,\cosh^2\vartheta
\end{aligned}\right.
$$
and, using \eqref{para-b}, we obtain explicit parametrizations of helix timelike surfaces that depend only of the hyperbolic angle $\vartheta$.
\begin{figure}[!h]
\begin{center}
\vspace{-.3 cm}
\includegraphics[width=0.1\linewidth]{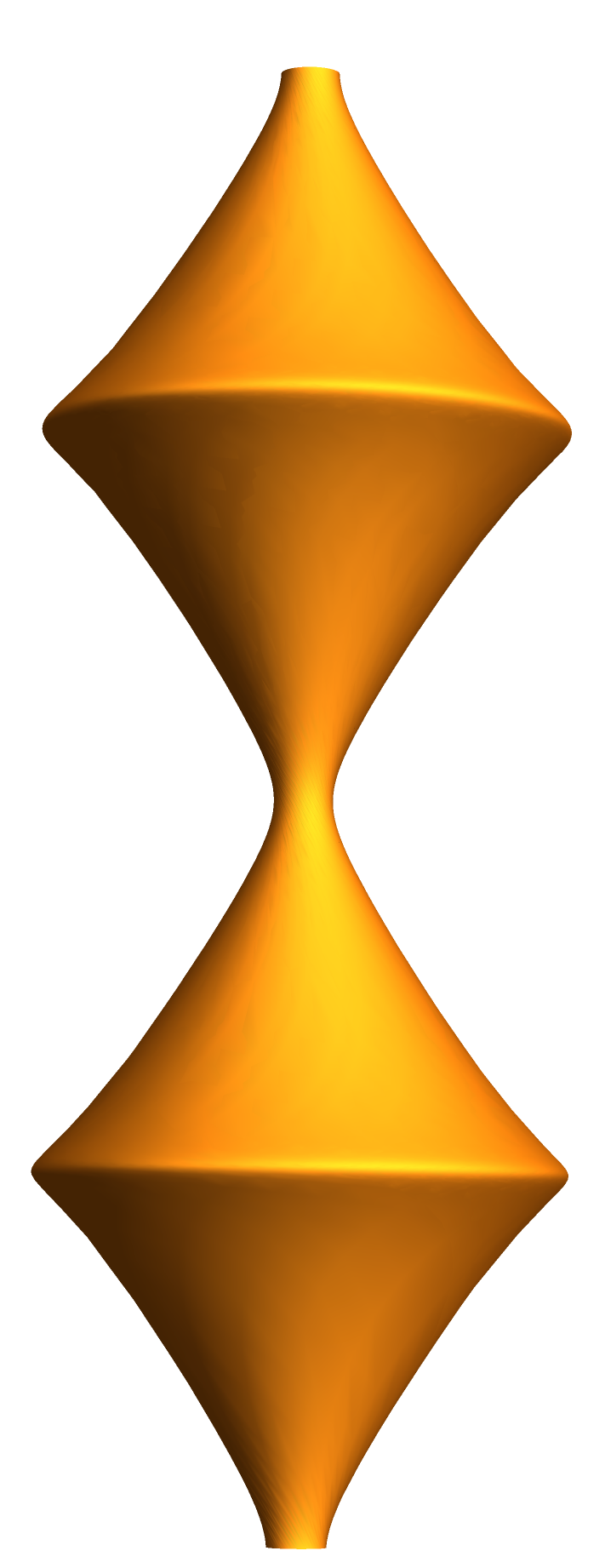} \qquad 
\includegraphics[width=0.1\linewidth]{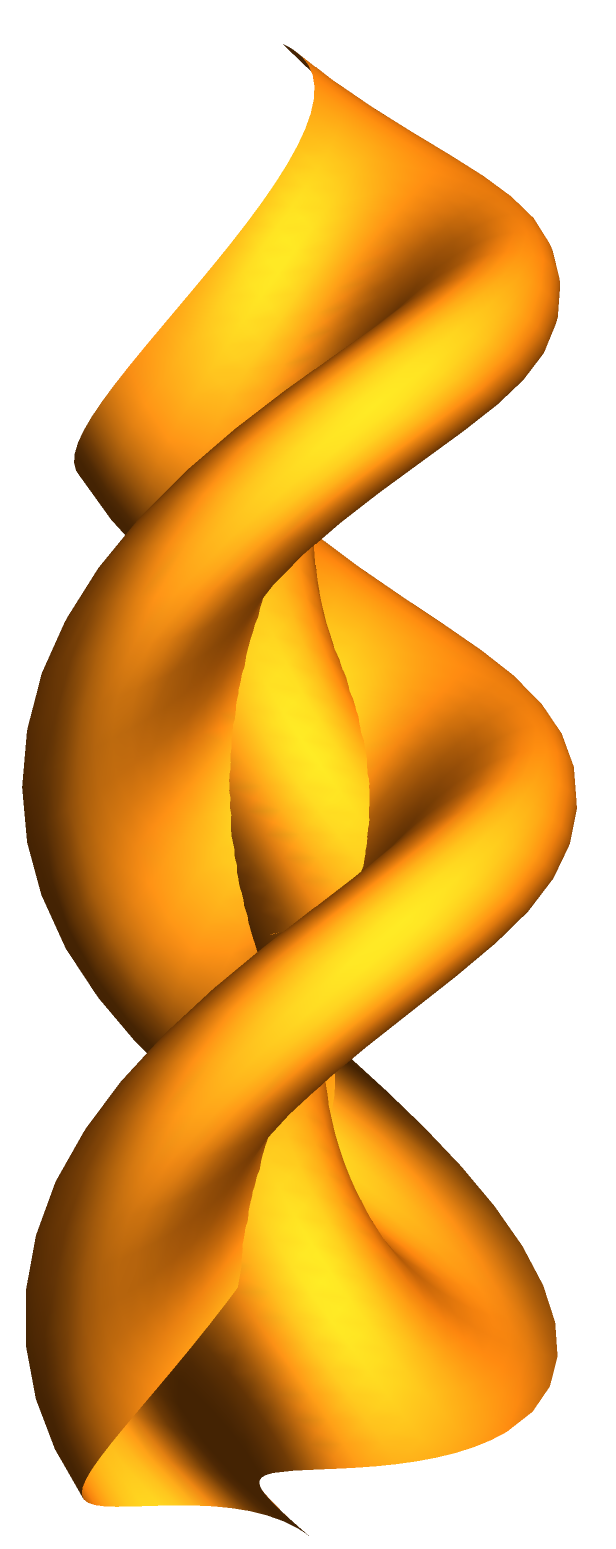}\qquad 
\includegraphics[width=0.15\linewidth]{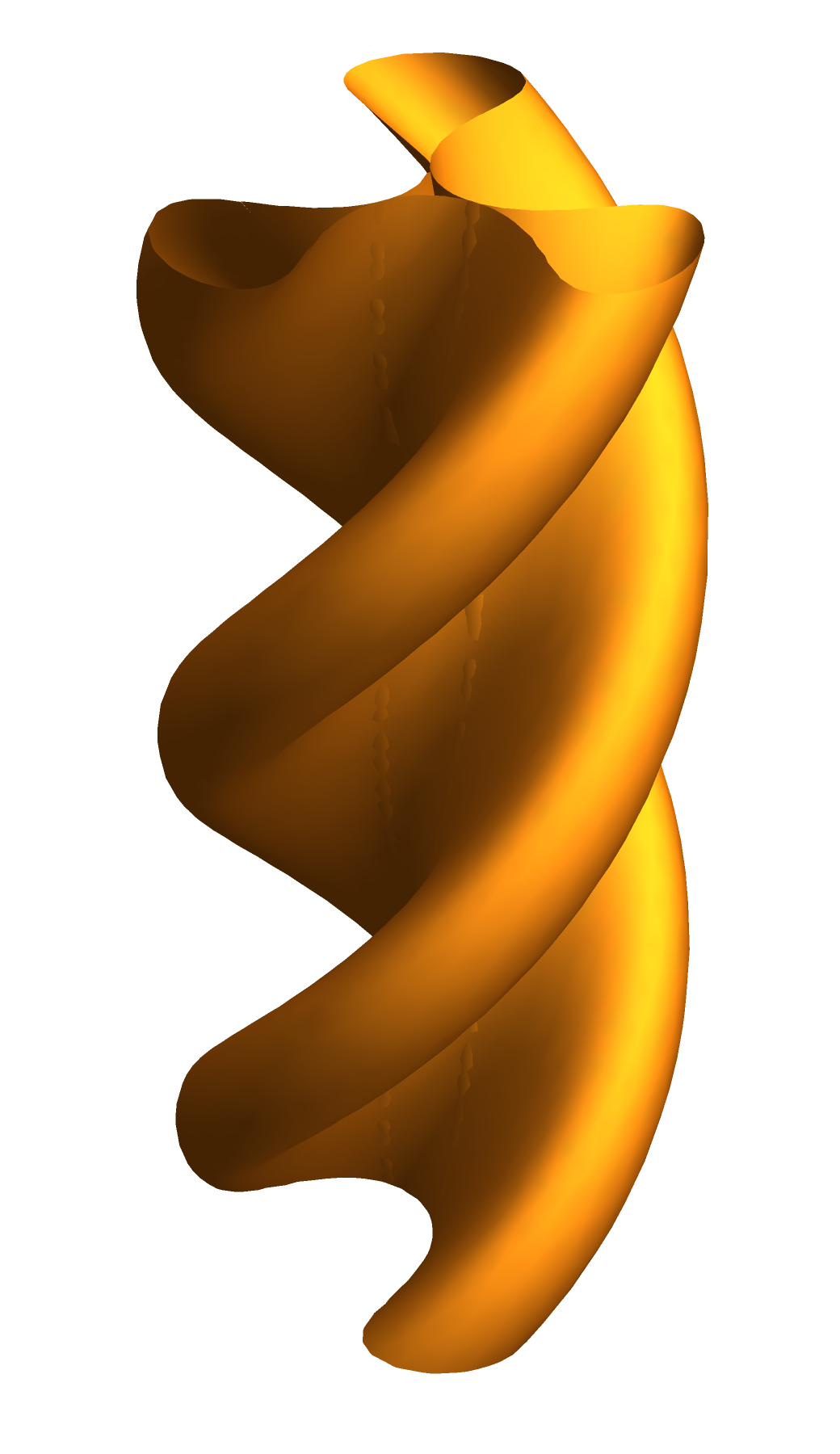}\qquad 
\includegraphics[width=0.23\linewidth]{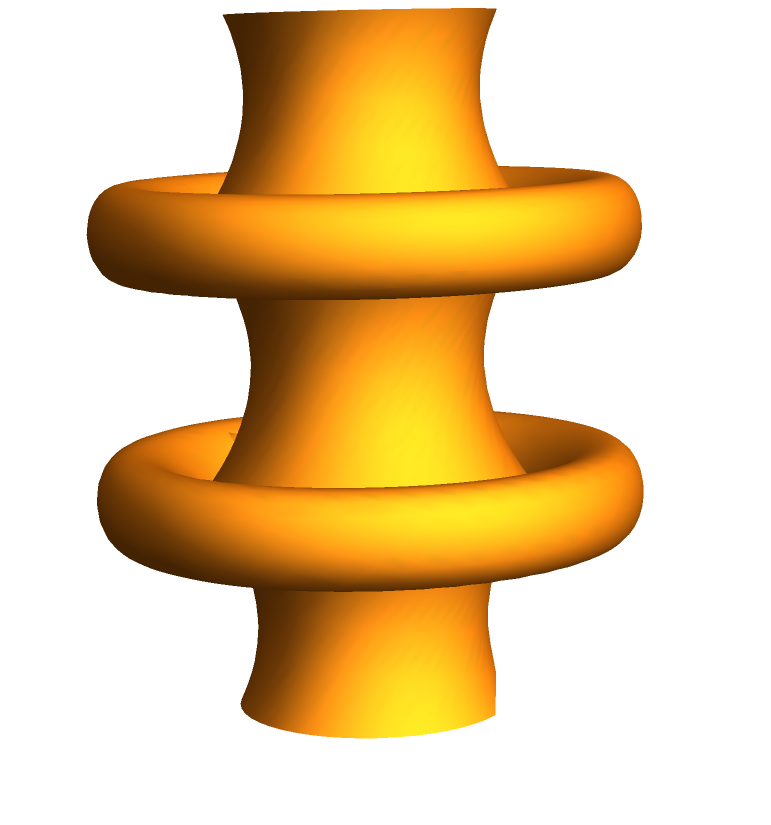}
\end{center}
\vspace{-.2 cm}
\caption{\it  Constant angle timelike surfaces for $\vartheta =\pi/3$, $\vartheta =\pi/4$, $\vartheta =\pi/6$  and $\vartheta =\pi/8$.}	
\label{fig-tempo}
\end{figure}\\
\end{example}

\end{document}